\documentclass[oneside,english]{amsart}
\usepackage[T1]{fontenc}
\usepackage[latin9]{inputenc}
\usepackage{verbatim}
\usepackage{textcomp}
\usepackage{enumitem}
\usepackage{amsthm}
\usepackage{amssymb}

\makeatletter
\numberwithin{equation}{section}
\numberwithin{figure}{section}
\theoremstyle{plain}
\newtheorem{thm}{\protect\theoremname}[section]
  \theoremstyle{plain}
  \newtheorem{fact}[thm]{\protect\factname}
  \theoremstyle{definition}
  \newtheorem{example}[thm]{\protect\examplename}
  \theoremstyle{definition}
  \newtheorem{defn}[thm]{\protect\definitionname}
  \theoremstyle{remark}
  \newtheorem{rem}[thm]{\protect\remarkname}
  \theoremstyle{plain}
  \newtheorem{lem}[thm]{\protect\lemmaname}
  \theoremstyle{plain}
  \newtheorem{prop}[thm]{\protect\propositionname}
\newenvironment{lyxlist}[1]
{\begin{list}{}
{\settowidth{\labelwidth}{#1}
 \setlength{\leftmargin}{\labelwidth}
 \addtolength{\leftmargin}{\labelsep}
 }}
{\end{list}}
 \newlist{casenv}{enumerate}{4}
 \setlist[casenv]{leftmargin=*,align=left,widest={iiii}}
 \setlist[casenv,1]{label={{\itshape\ \casename} \arabic*.},ref=\arabic*}
 \setlist[casenv,2]{label={{\itshape\ \casename} \roman*.},ref=\roman*}
 \setlist[casenv,3]{label={{\itshape\ \casename\ \alph*.}},ref=\alph*}
 \setlist[casenv,4]{label={{\itshape\ \casename} \arabic*.},ref=\arabic*}
  \theoremstyle{plain}
  \newtheorem{cor}[thm]{\protect\corollaryname}
  \theoremstyle{definition}
  \newtheorem{problem}[thm]{\protect\problemname}
  \theoremstyle{remark}
  \newtheorem{claim}[thm]{\protect\claimname}

\usepackage{multicol}
\usepackage{a4wide}
\usepackage{amssymb}
\usepackage{url}
\makeatother

\usepackage{babel}
  \providecommand{\claimname}{Claim}
  \providecommand{\corollaryname}{Corollary}
  \providecommand{\definitionname}{Definition}
  \providecommand{\examplename}{Example}
  \providecommand{\factname}{Fact}
  \providecommand{\lemmaname}{Lemma}
  \providecommand{\problemname}{Problem}
  \providecommand{\propositionname}{Proposition}
  \providecommand{\remarkname}{Remark}
 \providecommand{\casename}{Case}
\providecommand{\theoremname}{Theorem}

\begin{document}
\def\Ind#1#2{#1\setbox0=\hbox{$#1x$}\kern\wd0\hbox to 0pt{\hss$#1\mid$\hss}
\lower.9\ht0\hbox to 0pt{\hss$#1\smile$\hss}\kern\wd0}
\def\Notind#1#2{#1\setbox0=\hbox{$#1x$}\kern\wd0\hbox to 0pt{\mathchardef
\nn="3236\hss$#1\nn$\kern1.4\wd0\hss}\hbox to 0pt{\hss$#1\mid$\hss}\lower.9\ht0
\hbox to 0pt{\hss$#1\smile$\hss}\kern\wd0}
\def\indi{\mathop{\mathpalette\Ind{}}}
\def\nindi{\mathop{\mathpalette\Notind{}}}
\def\bdd {bdd}

\global\long\def\acl{\operatorname{acl}}

\global\long\def\cl{\operatorname{cl}}

\global\long\def\Avg{\operatorname{Avg}}

\global\long\def\Sk{\operatorname{Sk}}

\global\long\def\inp{\operatorname{inp}}

\global\long\def\dprk{\operatorname{dprk}}

\global\long\def\ind{\operatorname{\indi}}

\global\long\def\nind{\operatorname{\nindi}}

\global\long\def\ist{\operatorname{ist}}

\global\long\def\Aut{\operatorname{Aut}}

\global\long\def\M{\operatorname{\mathbb{M}}}

\global\long\def\NTP{\operatorname{NTP}}

\global\long\def\NIP{\operatorname{NIP}}

\global\long\def\IP{\operatorname{IP}}

\global\long\def\TP{\operatorname{TP}}

\global\long\def\tp{\operatorname{tp}}

\global\long\def\NSOP{\operatorname{NSOP}}

\global\long\def\bdn{\operatorname{bdn}}

\global\long\def\vstp{\operatorname{vstp}}

\global\long\def\lstp{\operatorname{Lstp}}

\global\long\def\sf{\operatorname{sf}}

\global\long\def\T{\operatorname{\mathbb{T}}}

\global\long\def\Q{\operatorname{\mathbb{Q}}}

\global\long\def\ot{\operatorname{ot}}

\global\long\def\Ded{\operatorname{ded}}

\global\long\def\Sym{\operatorname{Sym}}

\global\long\def\ded{\operatorname{ded}}

\global\long\def\cof{\operatorname{cf}}

\global\long\def\pp{\operatorname{pp}}

\global\long\def\Reg{\operatorname{Reg}}

\global\long\def\Th{\operatorname{Th}}

\global\long\def\tcof{\operatorname{tcf}}

\global\long\def\ran{\operatorname{ran}}

\title{On the number of Dedekind cuts and two-cardinal models of dependent
theories}

\author{Artem Chernikov and Saharon Shelah}

\thanks{The first author has received funding from the European Research
Council under the European Union\textquoteright{}s Seventh Framework
Programme {[}FP7/2007-2013{]} under grant agreement n\textdegree{}
238381 and from the ERC Grant Agreement No. 291111.}

\thanks{The second author would like to thank the Israel Science Foundation
for partial support of this research (Grant no. 1053/11). Publication
1035 on his list.}

\keywords{Dedekind cuts, linear orders, trees, cardinal arithmetic, PCF, two-cardinal
models, omitting types, dependent theories, NIP}
\begin{abstract}
For an infinite cardinal $\kappa$, let $\ded\kappa$ denote the supremum
of the number of Dedekind cuts in linear orders of size $\kappa$.
It is known that $\kappa<\ded\kappa\leq2^{\kappa}$ for all $\kappa$
and that $\ded\kappa<2^{\kappa}$ is consistent for any $\kappa$
of uncountable cofinality. We prove however that $2^{\kappa}\leq\ded\left(\ded\left(\ded\left(\ded\kappa\right)\right)\right)$
always holds. Using this result we calculate the Hanf numbers for
the existence of two-cardinal models with arbitrarily large gaps and
for the existence of arbitrarily large models omitting a type in the
class of countable dependent first-order theories. Specifically, we
show that these bounds are as large as in the class of all countable
theories.
\end{abstract}
\maketitle

\section{Introduction}

For an infinite cardinal $\kappa$, let 
\[
\ded\kappa=\sup\left\{ \left|I\right|:I\mbox{ is a linear order with a dense subset of size }\leq\kappa\right\} \mbox{.}
\]
In general the supremum need not be attained. Let $I$ be a linear
order and let $\mathfrak{c}=\left(I_{1},I_{2}\right)$ be a cut of
$I$ (i.e. $I=I_{1}\cup I_{2}$, $I_{1}\cap I_{2}=\emptyset$ and
$i_{1}<i_{2}$ for all $i_{1}\in I_{1},i_{2}\in I_{2}$). By \emph{cofinality
of $\mathfrak{c}$ from the left} (respectively, \emph{from the right})
we mean the cofinality of the linear order induced on $I_{1}$ (resp.
the cofinality of $I_{2}^{*}$, that is $I_{2}$ with the order reversed).
\begin{fact}
\label{fac: equivalent definitions of ded}The following cardinalities
are the same, see e.g. \cite[Proposition 6.5]{NFSpectra}:
\begin{enumerate}
\item $\ded\kappa$,
\item $\sup\left\{ \lambda:\mbox{exists a linear order }I\mbox{ of size }\leq\kappa\mbox{ with }\lambda\mbox{ cuts}\right\} $,
\item $\sup\{\lambda:$ exists a regular $\mu$ and a linear order of size
$\leq\kappa$ with $\lambda$ cuts of cofinality $\mu$ both from
the left and from the right$\}$,
\item $\sup\left\{ \lambda:\mbox{ exists a regular }\mu\mbox{ and a tree }T\mbox{ of size }\leq\kappa\mbox{ with }\lambda\mbox{ branches of length }\mu\right\} $.
\end{enumerate}
\end{fact}
It is well-known that $\kappa<\ded\kappa\leq\left(\ded\kappa\right)^{\aleph_{0}}\leq2^{\kappa}$
(for the first inequality, let $\mu$ be minimal such that $2^{\mu}>\kappa$,
and consider the tree $2^{<\mu}$) and that $\ded\aleph_{0}=2^{\aleph_{0}}$
(as $\mathbb{Q}\subseteq\mathbb{R}$ is dense). Thus $\ded\kappa=\left(\ded\kappa\right)^{\aleph_{0}}=2^{\kappa}$
for all $\kappa$ in a model with GCH. Moreover, Baumgartner \cite{Baum}
had shown that if $2^{\kappa}=\kappa^{+n}$ (i.e. the $n$th successor
of $\kappa$) for some $n\in\omega$, then $\ded\kappa=2^{\kappa}$.
On the other hand, for any $\kappa$ of uncountable cofinality Mitchell
\cite{Mitchell} had proven that consistently $\ded\left(\kappa\right)<2^{\kappa}$.
Besides, in \cite[Section 6]{NFSpectra} it is demonstrated that for
some $\kappa$ it is consistent that $\ded\kappa<\left(\ded\kappa\right)^{\aleph_{0}}$
(but it is still open if both inequalities $\ded\kappa\leq\left(\ded\kappa^{\aleph_{0}}\right)\leq2^{\kappa}$
can be strict simultaneously). The importance of the function $\ded\kappa$
from the model-theoretic point of view is largely due to the following
fact:
\begin{fact}
\label{fac: 6 stab functions}\cite{KeislerStabF,MR1083551} Let $T$
be a complete first-order theory in a countable language $L$. For
a model $M$ of $T$, $S_{1}\left(M\right)$ denotes the space of
$1$-types over $M$ (i.e. the space of ultrafilters on the Boolean
algebra of definable subsets of $M$). Define $f_{T}\left(\kappa\right)=\sup\left\{ \left|S_{T}\left(M\right)\right|:M\models T,\left|M\right|=\kappa\right\} $.
Then for any countable $T$, $f_{T}$ is one of the following functions:
$\kappa$, $\kappa+2^{\aleph_{0}}$, $\kappa^{\aleph_{0}}$, $\ded\kappa$,
$\left(\ded\kappa\right)^{\aleph_{0}}$ or $2^{\kappa}$ (and each
of these functions occurs for some $T$).
\end{fact}
In the first part of the paper we prove that $2^{\kappa}\leq\ded\left(\ded\left(\ded\left(\ded\kappa\right)\right)\right)$
holds for any $\kappa$. Our proof uses results from the PCF theory
of the second author. Optimality of this bound remains open. Moreover,
with two extra iterations we can ensure that the supremums are attained.
I.e., for any cardinal $\kappa$ there are linear orders $I_{0},\ldots,I_{6}$
such that $\left|I_{0}\right|\leq\kappa,2^{\kappa}\leq\left|I_{6}\right|$
and for every $i<6$, the number of Dedekind cuts in $I_{i}$ is at
least $\left|I_{i+1}\right|$.

~

In the second part of the paper we apply these results to questions
about cardinal transfer. Fix a complete first-order theory $T$ in
a countable language $L$, with a distinguished predicate $P\left(x\right)$
from $L$. Given two cardinals $\kappa\geq\lambda\geq\aleph_{0}$
we say that $M\models T$ is a $\left(\kappa,\lambda\right)$-model
if $\left|M\right|=\kappa$ and $\left|P\left(M\right)\right|=\lambda$.
A classical question in model theory is to determine implications
between existence of two-cardinal models for different pairs of cardinals.
It was studied by Vaught, Chang, Morley, Shelah and others.
\begin{fact}
(Vaught) Assume that for some $\kappa$, $T$ admits a $\left(\beth_{n}\left(\kappa\right),\kappa\right)$-model
for all $n\in\omega$. Then $T$ admits a $\left(\kappa',\lambda'\right)$-model
for any $\kappa'\geq\lambda'$.
\end{fact}
Vaught's theorem is optimal:
\begin{example}
\label{ex: optimality of Vaught}Fix $n\in\omega$, and consider a
structure $M$ in the language $L=\{P_{0}\left(x\right),\ldots,P_{n}\left(x\right),\in_{0},\ldots,\in_{n-1}\}$
in which $P_{0}\left(M\right)=\omega$, $P_{i+1}\left(M\right)$ is
the set of subsets of $P_{i}\left(M\right)$, and $\in_{i}\subseteq P_{i}\times P_{i+1}$
is the membership relation. Let $T=\Th\left(M\right)$. Then $M$
is a $\left(\beth_{n},\aleph_{0}\right)$-model of $T$, but it is
easy to see by ``extensionality'' that for any $M'\models T$ we
have $\left|M'\right|\leq\beth_{n}\left(\left|P_{0}\left(M'\right)\right|\right)$.
\end{example}
However, the theory in the example is wild from the model theoretic
point of view, and stronger transfer principles hold for tame classes
of theories.
\begin{fact}
\label{fac: stable and o-min}
\begin{enumerate}
\item \cite{Lachlan2Card} If $T$ is stable and admits a $\left(\kappa,\lambda\right)$-model
for some $\kappa>\lambda$, then it admits a $\left(\kappa',\lambda'\right)$-model
for any $\kappa'\geq\lambda'$.
\item \cite{Bays2Card} If $T$ is $o$-minimal and admits a $\left(\kappa,\lambda\right)$-model
for some $\kappa>\lambda$, then it admits a $\left(\kappa',\lambda'\right)$-model
for any $\kappa'\geq\lambda'$.
\end{enumerate}
\end{fact}
For further two-cardinal results for stable theories see \cite[Ch. V, §6]{MR1083551}
and also \cite{BerensteinShami}.

An important class of theories containing both the stable and the
o-minimal theories is the class of \emph{dependent} theories (also
called NIP theories in the literature) introduced by the second author
\cite{MR1083551}. In the countable case, dependent theories can be
defined as those theories for which $f_{T}\left(\kappa\right)\leq\left(\ded\kappa\right)^{\aleph_{0}}$
(see Fact \ref{fac: 6 stab functions}, and see Section \ref{sec: no beth_omega for dependent T}
for a combinatorial definition). Recently dependent theories have
attracted a lot of attention both in purely model theoretic work on
generalizing the machinery of stable theories (e.g. \cite{ShelahDependentCont,900,950,ExtDefI,ExtDef2}),
and due to the analysis of some important algebraic examples \cite{HruPi,HHM}.

It is easy to see that the theory in Example \ref{ex: optimality of Vaught}
is not dependent, but also that a complete analogue of Fact \ref{fac: stable and o-min}
cannot hold for dependent theories: consider the theory of $\left(\mathbb{R},<\right)$
expanded by a predicate naming $\mathbb{Q}$. In Section \ref{sec: no beth_omega for dependent T}
we show that in fact the situation for dependent theories is not better
than for arbitrary theories, in contrast to the stable and o-minimal
cases. Namely, for every $n<\omega$ we construct a \emph{dependent}
theory $T_{n}$ which has a $(\beth_{m},\aleph_{0})$-model for all
$m<n$, but does not have a $\left(\beth_{\omega},\aleph_{0}\right)$-model.
In Section \ref{sec: Hanf number for omitting types} we elaborate
on this example and show that the Hanf number for omitting a type
is again the same for countable dependent theories as for arbitrary
theories --- unlike in the stable \cite{HrushovskiShelah} and in
the $o$-minimal \cite{Marker} cases. Examples which we construct
add to the list of dependent theories \cite{Sh946,Sh975} demonstrating
that the principle ``dependent = stable + linear order'' has only
limited applicability.

\section{\label{sec: on the number of dedekind cuts}On the number of Dedekind
cuts}

\subsection{On $\pp_{\kappa}\left(\lambda\right)$}

We summarize some facts from the PCF theory of the second author (see
also \cite[Chapter 9]{IntroToCardArithm} for an exposition).
\begin{defn}
\label{def: leq plus}Given a set of cardinals $A$ and a cardinal
$\lambda$, we will write $\sup^{+}\left(A\right)=\min\{\mu:\forall\nu\in A,\nu<\mu\}$
and $\lambda\leq^{+}\sup\left(A\right)$ if either $\lambda<\sup\left(A\right)$,
or $\lambda=\sup\left(A\right)$ and $\lambda\in A$.
\end{defn}

\begin{defn}
\cite[II.§1]{ShCardinalArithmetic} For $\cof\lambda\leq\kappa<\lambda$
let 
\[
A=\left\{ \cof\left(\prod a/\mathcal{F}\right)\,:\, a\subset\Reg\land\sup\left(a\right)=\lambda\land\left|a\right|\leq\kappa\land\mathcal{F}\mbox{ is an ultrafilter on }a\land\mathcal{F}\cap I_{b}\left(a\right)=\emptyset\right\} \mbox{,}
\]

where $\Reg$ is the class of regular cardinals, and for a set $B$
of ordinals with $\sup\left(B\right)\notin B$, $I_{b}\left(B\right)=\left\{ X\subseteq B:\exists\beta\in B\, X\subseteq\beta\right\} $
denotes the ideal of bounded subsets of $B$. Then we define $\pp_{\kappa}\left(\lambda\right)=\sup\left(A\right)$
and $\pp_{\kappa}^{+}\left(\lambda\right)=\sup^{+}\left(A\right)$
(where ``$\pp$'' stands for ``pseudo-power'').
\end{defn}
Equivalently (see e.g \cite[Lemma 9.1.1]{IntroToCardArithm}), for
$\cof\lambda\leq\kappa<\lambda$ one has 
\[
\pp_{\kappa}\left(\lambda\right)=\sup\left\{ \tcof\left(\prod_{i<\kappa}\lambda_{i}/I,<_{I}\right):\lambda_{i}=\cof\lambda_{i}<\lambda=\sup_{i<\kappa}\lambda_{i}\land I\mbox{ is an ideal on }\kappa\land I_{b}\left(\kappa\right)\subseteq I\right\} \mbox{,}
\]
where $<_{I}$ is the lexicographic ordering modulo $I$ and for a
partial order $P$, $\tcof\left(P\right)=\kappa$ when there are $\left\langle p_{i}:i<\kappa\right\rangle $
in $P$ such that $\kappa=\cof\kappa$ and $\bigwedge_{i<j}\left(p_{i}<p_{j}\right)$
and $\forall p\in P\left(\bigvee_{i<\kappa}p\leq p_{i}\right)$ (true
cofinality may not exist). We recall that $\Gamma\left(\theta,\sigma\right)=\left\{ I:\mbox{for some cardinal }\theta_{I}<\theta,\, I\mbox{ is a \ensuremath{\sigma}-complete ideal on \ensuremath{\theta_{I}}}\right\} $
and $\Gamma\left(\theta\right)=\Gamma\left(\theta^{+},\theta\right)$.
Then $\pp_{\Gamma\left(\theta,\sigma\right)}\left(\lambda\right)$
is defined in the same way as $\pp_{\kappa}\left(\lambda\right)$
but the supremum is taken only over ideals from $\Gamma\left(\theta,\sigma\right)$.
\begin{fact}
\label{fac: general properties of pp} See e.g. \cite[Chapter 9]{IntroToCardArithm}:
\begin{enumerate}
\item $\lambda<\pp_{\kappa}\left(\lambda\right)\leq\lambda^{\kappa}$ and
if $\cof\lambda=\kappa>\aleph_{0}$ and $\lambda$ is $\kappa$-strong
(i.e. $\rho^{\kappa}<\lambda$ for all $\rho<\lambda$), then $\pp_{\kappa}\left(\lambda\right)=\lambda^{\kappa}$.
In particular $\pp_{\kappa}\left(\lambda\right)=\lambda^{\kappa}$
holds for any strong limit $\lambda$ with uncountable cofinality
$\kappa$.
\item For any $\theta$ we have $\pp_{\Gamma\left(\theta\right)}\left(\lambda\right)\leq\pp_{\theta}\left(\lambda\right)$
and $\pp_{\Gamma\left(\theta^{+},2\right)}\left(\lambda\right)=\pp_{\theta}\left(\lambda\right)$.
\end{enumerate}
\end{fact}

\begin{fact}
\label{fac: PCF} 

\begin{enumerate}
\item \label{enu: finding inaccessible mu}\cite[4.3]{Sh410} Assume:

\begin{itemize}
\item $\lambda$ is regular, uncountable,
\item $\kappa<\lambda$ implies $2^{\kappa}<2^{\lambda}$,
\item for some regular $\chi\leq2^{\lambda}$ there is no tree of cardinality
$\lambda$ with $\geq\chi$-many branches of length $\lambda$.
\end{itemize}

Then $2^{<\lambda}<2^{\leq\lambda}$, and for some $\mu\in\left(\lambda,2^{<\lambda}\right]$
with $\cof\mu=\lambda$:
\begin{enumerate}
\item for every regular $\chi$ in $\left(2^{<\lambda},2^{\lambda}\right]$
there is a linear order of cardinality $\chi$ with a dense subset
of cardinality $\mu$ (the linear order is $\left(T_{\chi},<_{\mbox{lx}}\right)$,
where $T_{\chi}\subseteq2^{<\mu}$ has $\leq\mu$ nodes and $\geq\chi$-many
branches of length $\lambda$),
\item $\pp_{\Gamma\left(\lambda\right)}\left(\mu\right)=2^{\lambda}$, 
\item $\mu$ is $\left(\lambda,\lambda^{+},2\right)$-inaccessible, i.e.
(see \cite[3.2]{Sh410}) for any $\mu'$ such that $\lambda<\mu'<\mu\,\land\,\mbox{\ensuremath{\cof}}\mu'\leq\lambda$
we have $\pp_{\Gamma\left(\lambda^{+},2\right)}\left(\mu'\right)<\mu$,
which in view of Fact \ref{fac: general properties of pp} implies
$\pp_{\lambda}\left(\mu'\right)<\mu$.
\end{enumerate}
\item \label{enu: finding a sequence of mu's}\cite[Claim 3.4]{Sh430} Assume
that $\theta_{n+1}=\min\left\{ \theta\,:\,2^{\theta}>2^{\theta_{n}}\right\} $
for $n<\omega$ and $\sum_{n<\omega}\theta_{n}<2^{\theta_{0}}$ (so
$\theta_{n+1}$ is regular, $\theta_{n+1}>\theta_{n}$). Then for
infinitely many $n<\omega$, for some $\mu_{n}\in\left[\theta_{n},\theta_{n+1}\right)$
(so $2^{\mu_{n}}=2^{\theta_{n}}$) we have: for every regular $\chi\leq2^{\theta_{n}}$
there is a tree of cardinality $\mu_{n}$ with $\geq\chi$-many branches
of length $\theta_{n}$.
\item \label{enu: inverting pp's}\cite[II.2.3(2)]{ShCardinalArithmetic}
If $\lambda<\mu$ are singulars of cofinality $\leq\kappa$ (and $\kappa<\lambda$)
and $\pp_{\kappa}\left(\lambda\right)\geq\mu$ then $\pp_{\kappa}\left(\mu\right)\leq^{+}\pp_{\kappa}\left(\lambda\right)$.
\end{enumerate}
\end{fact}
\begin{rem}
See \cite{GitikShelah} concerning optimality of these results.
\end{rem}

\subsection{Bounding exponent by iterated $\ded$}
\begin{defn}
By induction on the ordinal $\alpha$ we define a strictly increasing
sequence of ordinals $\gimel_{\alpha}$ such that:
\begin{itemize}
\item If $\alpha=0$, then $\gimel_{\alpha}=\aleph_{0}$.
\item If $\alpha=\beta+1$, then $\gimel_{\alpha}=\min\left\{ \gimel:2^{\gimel}>2^{\gimel_{\beta}}\right\} $.
\item If $\alpha$ is limit, then $\gimel_{\alpha}=\sum\left\{ \gimel_{\beta}:\beta<\alpha\right\} $.
\end{itemize}
\end{defn}
\begin{lem}
\label{lem: exp vs ded on gimels}For any ordinal $\alpha$, $2^{\gimel_{\alpha+1}}\leq^{+}\ded\left(2^{\gimel_{\alpha}}\right)$.\end{lem}
\begin{proof}
$2^{<\gimel_{\alpha+1}}$ is a tree with $2^{\gimel_{\alpha+1}}$
branches and $\leq\sum\left\{ 2^{\left|\beta\right|}:\beta<\gimel_{\alpha+1}\right\} $
nodes. But if $\beta<\gimel_{\alpha+1}$, then $2^{\beta}\leq2^{\gimel_{\alpha}}$
and $\gimel_{\alpha+1}\leq2^{\gimel_{\alpha}}$ by the definition
of $\gimel$'s, so the number of nodes is bounded by $2^{\gimel_{\alpha}}$.\end{proof}
\begin{prop}
\label{prop: two ded iterations} Assume that $\gimel_{\alpha+k}\leq2^{\gimel_{\alpha}}$
for some $k\in\omega$. Then for some $m\leq k$:
\begin{itemize}
\item $\ded\left(2^{\gimel_{\alpha}}\right)\geq2^{\gimel_{\alpha+m}}$,
\item $\ded\left(2^{\gimel_{\alpha+m}}\right)\geq2^{\gimel_{\alpha+k}}$.
\end{itemize}
\end{prop}
\begin{proof}
We follow the proof of \cite[Claim 3.4]{Sh430}. Let $\theta_{n}=\gimel_{\alpha+n}$
for $n\leq k$. Note that $\theta_{n+1}$ is regular and $\theta_{n+1}>\theta_{n}$.
We define:
\begin{lyxlist}{00.00.0000}
\item [{$\left(*\right)_{\theta_{n}}$}] for every regular $\chi\leq2^{\theta_{n}}$
there is a tree of cardinality $\theta_{n}$ with $\geq\chi$-many
branches of length $\theta_{n}$.
\end{lyxlist}
Let $S_{0}=\left\{ 0<n\leq k:\left(*\right)_{\theta_{n}}\mbox{ fails}\right\} $. 

By Fact \ref{fac: PCF}(\ref{enu: finding inaccessible mu}) with
$\lambda=\theta_{n}$ and the definitions of $S_{0}$ and of the $\gimel$'s
it follows that for each $n\in S_{0}$ there is $\mu_{n}$ such that: 
\begin{lyxlist}{00.00.0000}
\item [{$\left(\alpha\right)_{n}$}] $\theta_{n}=\cof\mu_{n}<\mu_{n}\leq2^{<\theta_{n}}=2^{\theta_{n-1}}$(as
$2^{<\theta_{n}}\leq\theta_{n}\times2^{\theta_{n-1}}\leq2^{\theta_{0}}\times2^{\theta_{n-1}}\leq2^{\theta_{n-1}}$).
\item [{$\left(\beta\right)_{n}$}] $\pp_{\theta_{n}}\left(\mu_{n}\right)=\pp_{\Gamma\left(\theta_{n}\right)}\left(\mu_{n}\right)=2^{\theta_{n}}$
(as $\pp_{\Gamma\left(\theta_{n}\right)}\left(\mu_{n}\right)=2^{\theta_{n}}$
by Fact \ref{fac: PCF}(\ref{enu: finding inaccessible mu})(b), and
$\pp_{\Gamma\left(\theta_{n}\right)}\left(\mu_{n}\right)\leq\pp_{\theta_{n}}\left(\mu_{n}\right)\leq\mu_{n}^{\theta_{n}}\leq\left(2^{\theta_{n-1}}\right)^{\theta_{n}}\leq2^{\theta_{n}}$
by Fact \ref{fac: general properties of pp}).
\item [{$\left(\gamma\right)_{n}$}] For any $\mu'$ we have that $\theta_{n}<\mu'<\mu_{n}\,\land\,\mbox{\ensuremath{\cof}}\mu'\leq\theta_{n}$
implies $\pp_{\Gamma\left(\lambda^{+},2\right)}\left(\mu'\right)<\mu_{n}$
(by Fact \ref{fac: PCF}(\ref{enu: finding inaccessible mu})(c)).%

\item [{$\left(\delta\right)_{n}$}] $\ded\left(\mu_{n}\right)\geq2^{\theta_{n}}$
(as for any regular $\chi\leq2^{\theta_{n}}$ there is linear order
of cardinality $\geq\chi$ with a dense subset of size $\mu_{n}$
by Fact \ref{fac: PCF}(\ref{enu: finding inaccessible mu})(a)).
\end{lyxlist}

Let $S_{1}=\left\{ n\in S_{0}:\mu_{n}\geq2^{\gimel_{\alpha}}\right\} $.
Then we have the following claims.

~
\begin{lyxlist}{00.00.0000}
\item [{$\left(*\right)_{1}$}] If $n\leq k$ and $n\notin S_{0}$ then
$\ded\left(2^{\gimel_{\alpha}}\right)\geq2^{\gimel_{\alpha+n}}$.
\end{lyxlist}
\emph{Proof}. By the definition of $S_{0}$ and of $\theta_{n}$ it
follows that $\ded\left(\theta_{n}\right)\geq2^{\gimel_{\alpha+n}}$
(taking supremum over trees corresponding to regular $\chi$'s less
or equal to $2^{\theta_{n}}$), and $\theta_{n}\leq2^{\gimel_{\alpha}}$
by assumption. Thus $\ded\left(2^{\gimel_{\alpha}}\right)\geq2^{\gimel_{\alpha+n}}$
as wanted.

~
\begin{lyxlist}{00.00.0000}
\item [{$\left(*\right)_{2}$}] If $n\leq k$ and $n\in S_{0}\setminus S_{1}$
then $\ded\left(2^{\gimel_{\alpha}}\right)\geq2^{\gimel_{\alpha+n}}$.
\end{lyxlist}
\emph{Proof}. By the definition of $S_{1}$ we have $\mu_{n}<2^{\gimel_{\alpha}}$.
On the other hand, as $n\in S_{0}$, we have $\ded\left(\mu_{n}\right)\geq2^{\theta_{n}}$
by $\left(\delta\right)_{n}$. Combining we get $\ded\left(2^{\gimel_{\alpha}}\right)\geq2^{\gimel_{\alpha+n}}$.%

~
\begin{lyxlist}{00.00.0000}
\item [{$\left(*\right)_{3}$}] If $n$ and $n+1$ are from $S_{1}$ then
$\mu_{n}>\mu_{n+1}$.
\end{lyxlist}
\emph{Proof}. By the assumption $\mu_{n}\geq2^{\gimel_{\alpha}}\geq\theta_{n+1}=\cof\theta_{n+1}$,
and in fact $\mu_{n}>\theta_{n+1}$ as they are of different cofinality. 

Assume that $\mu_{n}<\mu_{n+1}$. Then by Fact \ref{fac: PCF}(\ref{enu: inverting pp's})
with $\lambda=\mu_{n}$, $\mu=\mu_{n+1}$ and $\kappa=\theta_{n+1}$
(as $\max\left\{ \cof\mu_{n},\cof\mu_{n+1}\right\} =\max\left\{ \theta_{n},\theta_{n+1}\right\} <\min\left\{ \mu_{n},\mu_{n+1}\right\} $
by $\left(\alpha\right)_{n}$ and $\left(\alpha\right)_{n+1}$, and
$\pp_{\theta_{n+1}}\left(\mu_{n}\right)\geq\pp_{\Gamma\left(\theta_{n}\right)}\left(\mu_{n}\right)=2^{\theta_{n}}\geq\mu_{n+1}$)
we would get $\pp_{\theta_{n+1}}\left(\mu_{n+1}\right)\leq^{+}\pp_{\theta_{n+1}}\left(\mu_{n}\right)$. 

On the other hand by $\left(\gamma\right)_{n+1}$ we would get that
$\theta_{n+1}<\mu_{n}<\mu_{n+1}\,\land\,\mbox{\ensuremath{\cof}}\mu_{n}\leq\theta_{n+1}$
implies $\pp_{\theta_{n+1}}\left(\mu_{n}\right)<\mu_{n+1}\leq2^{\theta_{n+1}}=\pp_{\theta_{n+1}}\left(\mu_{n+1}\right)$
--- a contradiction. %
{} Thus we conclude that $\mu_{n}\geq\mu_{n+1}$, and in fact $\mu_{n}>\mu_{n+1}$
as they are of different cofinalities.

~

We try to define $m=\max\left\{ 0<n\leq k:n\notin S_{1}\right\} $.%

\begin{casenv}
\item $m$ not defined. So $S_{1}=\left\{ 1,\ldots,k\right\} $ (and we
may assume that $k\geq2$), hence $\mu_{1}>\ldots>\mu_{k}$ by $\left(*\right)_{3}$,
hence $\mu_{k}<\mu_{1}\leq2^{\theta_{0}}$. But by the definition
of $S_{1}$ actually $\mu_{k}\geq2^{\theta_{0}}$ --- a contradiction.
\item $m$ is well-defined. So $\left\{ m+1,\ldots,k\right\} \subseteq S_{1}$
hence as in Case 1 we have $\mu_{k}<\mu_{m+1}\leq2^{\theta_{m}}$
hence $\ded\left(2^{\gimel_{\alpha+m}}\right)\geq\ded\left(\mu_{k}\right)\geq2^{\gimel_{\alpha+k}}$
by $\left(\delta\right)_{k}$. Besides, $\ded\left(2^{\gimel_{\alpha}}\right)\geq2^{\gimel_{\alpha+m}}$
(by $\left(*\right)_{1}$ if $m\notin S_{0}$ and by $\left(*\right)_{2}$
if $m\in S_{1}\setminus S_{0}$) --- so we are done.
\end{casenv}
\end{proof}
\begin{prop}
\label{prop: two ded plus iterations} Assume that $\gimel_{\alpha+k}\leq2^{\gimel_{\alpha}}$
for some $k\in\omega$. Then for some $m\leq k$:
\begin{itemize}
\item $2^{\gimel_{\alpha+k}}\leq^{+}\ded\left(2^{\gimel_{\alpha+k-1}}\right)$,
\item $2^{\gimel_{\alpha+k-1}}\leq^{+}\ded\left(2^{\gimel_{\alpha+m}}\right)$,
\item $2^{\gimel_{\alpha+m}}\leq^{+}\ded\left(2^{\gimel_{\alpha+m-1}}\right)$,
\item $2^{\gimel_{\alpha+m-1}}\leq^{+}\ded\left(2^{\gimel_{\alpha}}\right)$.
\end{itemize}
\end{prop}
\begin{proof}
We modify the proof of Proposition \ref{prop: two ded iterations}.
We have:
\begin{lyxlist}{00.00.0000}
\item [{$\left(*\right)_{1}^{+}$}] If $n+1\leq k$ and $n+1\notin S_{0}$
then $\ded\left(2^{\gimel_{\alpha}}\right)\,^{+}\geq2^{\gimel_{\alpha+n}}$.
\end{lyxlist}
\emph{Proof.} As $\left(2^{\gimel_{\alpha+n}}\right)^{+}$ is regular,
$\left(2^{\gimel_{\alpha+n}}\right)^{+}\leq2^{\gimel_{\alpha+n+1}}$
and $\left(*\right)_{\theta_{n+1}}$ holds by the definition of $S_{0}$,
it follows that $\ded\left(\theta_{n+1}\right)\,^{+}\geq2^{\gimel_{\alpha+n}}$,
and $\theta_{n+1}\leq2^{\gimel_{\alpha}}$ by assumption. Thus $\ded\left(2^{\gimel_{\alpha}}\right)\,^{+}\geq2^{\gimel_{\alpha+n}}$
as wanted.

~
\begin{lyxlist}{00.00.0000}
\item [{$\left(*\right)_{2}^{+}$}] If $n+1\leq k$ and $n+1\in S_{0}\setminus S_{1}$
then $\ded\left(2^{\gimel_{\alpha}}\right)\geq2^{\gimel_{\alpha+n}}$.
\end{lyxlist}
\emph{Proof}. If $n+1\in S_{0}\setminus S_{1}$ then $\mu_{n+1}<2^{\gimel_{\alpha}}$
and $\ded\left(\mu_{n+1}\right)\,^{+}\geq2^{\theta_{n}}$ by $\left(\delta\right)_{n+1}$. 

~

Now in Case 1 we get a contradiction in the same way as before, so
we may assume that $m$ is well defined, i.e. $\left\{ m+1,\ldots,k\right\} \subseteq S_{1}$.
As before we get $\mu_{k}<\mu_{m+1}\leq2^{\theta_{m}}$, hence $\ded\left(2^{\gimel_{\alpha+m}}\right)\geq\ded\left(\mu_{k}\right)\,^{+}\geq2^{\gimel_{\alpha+k-1}}$
by $\left(\delta\right)_{k}$. Besides, $\ded\left(2^{\gimel_{\alpha}}\right)\,^{+}\geq2^{\gimel_{\alpha+m-1}}$
(by $\left(*\right)_{1}^{+}$ if $m\notin S_{0}$ and by $\left(*\right)_{2}^{+}$
if $m\in S_{1}\setminus S_{0}$). We can conclude by Lemma \ref{lem: exp vs ded on gimels}.
\end{proof}

Although, as it was already mentioned, it is consistent for $\kappa$
of uncountable cofinality that $\ded\kappa<2^{\kappa}$, we prove
(in ZFC) that these values are not so far apart and that four iterations
of $\ded$ are sufficient to get the exponent.
\begin{thm}
\label{thm: ded bigger than exponent} Let $\mu$ be an arbitrary
cardinal. Then there are $\lambda_{0},\ldots,\lambda_{4}$ such that:
\begin{enumerate}
\item $\lambda_{0}\leq\mu$,
\item $\lambda_{i+1}\leq\ded\left(\lambda_{i}\right)$ for $i<4$,
\item $2^{\mu}\leq\lambda_{4}$.
\end{enumerate}
\end{thm}
\begin{proof}
As the sequence of the $\gimel$'s is increasing, for some $\alpha$
we have $\gimel_{\alpha}\leq\mu<\gimel_{\alpha+1}$, so also $\alpha\leq\mu$.

First of all, for any ordinal $\beta$ with $\beta+\omega\leq\alpha$
and $2^{\gimel_{\beta}}>\gimel_{\beta+\omega}$ we have (by Fact \ref{fac: PCF}(\ref{enu: finding a sequence of mu's})
taking $\theta_{0}=\gimel_{\beta}$ and $\theta_{n}=\gimel_{\beta+n}$):
\begin{lyxlist}{00.00.0000}
\item [{$\odot_{1}$}] For infinitely many $\gamma\in\left[\beta,\beta+\omega\right)$
and arbitrary regular $\gimel\leq2^{\gimel_{\gamma}}$, there is a
tree $T$ with $\left|T\right|\in\left[\gimel_{\gamma},\gimel_{\gamma+1}\right)$
and at least $\gimel$-many branches of length $\gimel_{\gamma}$.
\end{lyxlist}
Let $\delta_{*}$ be the largest non-successor ordinal $\leq\alpha$,
so $\alpha=\delta_{*}+n_{*}$ for some $n_{*}<\omega$. We have:
\begin{lyxlist}{00.00.0000}
\item [{$\odot_{2}$}] There is a linear order $I$ of cardinality $\leq\mu$
with $\geq\sum\left\{ 2^{\gimel_{\beta}}\,:\,\beta<\delta_{*}\right\} $
Dedekind cuts.
\end{lyxlist}
(Indeed, if $\gimel_{\delta_{*}}$ is a strong limit cardinal then
$\sum\left\{ 2^{\gimel_{\beta}}\,:\,\beta<\delta_{*}\right\} \leq\mu$
and this is trivial. Otherwise, the demand $\gimel_{\beta+\omega}\leq2^{\gimel_{\beta}}<2^{\gimel_{\beta+1}}$
holds for every large enough $\beta<\delta_{*}$, so by $\odot_{1}$
and Fact \ref{fac: equivalent definitions of ded} we can conclude
by taking the sum of the corresponding linear orders and noting that
$\delta_{*}\leq\mu$).

Let $\lambda_{0}=\mu$, $\lambda_{1}=\sum\left\{ 2^{\gimel_{\beta}}\,:\,\beta<\delta_{*}\right\} $
and $\lambda_{2+n}=2^{\gimel_{\delta_{*}+n}}$ for $n\in\left\{ 0,\ldots,n_{*}\right\} $.
Note that $\lambda_{2+n_{*}}=2^{\gimel_{\alpha}}=2^{\mu}$. 

We have:
\begin{itemize}
\item $\lambda_{1}\leq^{+}\Ded\lambda_{0}$ (by $\odot_{2}$).
\item $\lambda_{2}\leq^{+}\Ded\lambda_{1}$ (as $2^{<\gimel_{\delta_{*}}}$
is a tree with $\sum\left\{ 2^{\kappa}\,:\,\kappa<\gimel_{\delta_{*}}\right\} =\sum\left\{ 2^{\gimel_{\beta}}\,:\,\beta<\delta_{*}\right\} =\lambda_{1}$
nodes and $2^{\gimel_{\delta_{*}}}=\lambda_{2}$ branches).
\item $\lambda_{2+n+1}\leq^{+}\Ded\left(\lambda_{2+n}\right)$ for $n<n_{*}$
(by Lemma \ref{lem: exp vs ded on gimels}).
\end{itemize}
If $\delta_{*}=\alpha$ then we are done as $\lambda_{2}=2^{\gimel_{\alpha}}=2^{\mu}$
(as $\mu<\gimel_{\alpha+1}$ and $\gimel_{\alpha+1}$ is smallest
with $2^{\gimel_{\alpha}}<2^{\gimel_{\alpha+1}}$), so assume $\delta_{*}=\alpha_{*}+n_{*}$
and $n_{*}>0$.

If $\gimel_{\delta_{*}+n_{*}}\leq2^{\gimel_{\delta_{*}}}$, then by
Proposition \ref{prop: two ded iterations} there is some $m\leq n_{*}$
such that $\lambda_{3}'=\ded\left(2^{\gimel_{\delta_{*}}}\right)\geq2^{\gimel_{\delta_{*}+m}}$
and $\lambda_{4}'=\ded\left(2^{\gimel_{\delta_{*}+m}}\right)\geq2^{\gimel_{\delta_{*}+n_{*}}}=2^{\gimel_{\alpha}}=2^{\mu}$.
It then follows that $\lambda_{0},\lambda_{1},\lambda_{2},\lambda_{3}',\lambda_{4}'$
are as wanted.

Otherwise $\gimel_{\delta_{*}+n_{*}}>2^{\gimel_{\delta_{*}}}$, and
let $n$ be the biggest such that $\gimel_{\delta_{*}+n_{*}}>2^{\gimel_{\delta_{*}+n}}$,
it follows that $n\leq n_{*}-1$. Then $\gimel_{\delta_{*}+n_{*}}\leq2^{\gimel_{\delta_{*}+n+1}}$
and again by Proposition \ref{prop: two ded iterations} we get some
$m$ such that:
\begin{itemize}
\item $\lambda_{0}''=2^{\gimel_{\delta_{*}+n}}<\gimel_{\delta_{*}+n_{*}}\leq\mu$,
\item $\lambda_{1}''=2^{\gimel_{\delta_{*}+n+1}}\leq^{+}\ded\left(2^{\gimel_{\delta_{*}+n}}\right)$
(by Lemma \ref{lem: exp vs ded on gimels}),
\item $\lambda_{2}''=2^{\gimel_{\delta_{*}+m}}\leq\ded\left(2^{\gimel_{\delta_{*}+n+1}}\right)$,
\item $2^{\mu}=2^{\gimel_{\delta_{*}+n_{*}}}\leq\lambda_{3}''=\ded\left(2^{\gimel_{\delta_{*}+m}}\right)$.
\end{itemize}
But then $\left\langle \lambda_{i}''\right\rangle _{i\leq3}$ are
as wanted.
\end{proof}
Similarly we have:
\begin{cor}
\label{cor: reaching exp by ded plus}Let $\mu$ be an arbitrary cardinal.
Then there are $\lambda_{0},\ldots,\lambda_{6}$ such that:
\begin{enumerate}
\item $\lambda_{0}\leq\mu$,
\item $\lambda_{i+1}\leq^{+}\Ded(\lambda_{i})$ for all $i<6$,
\item $2^{\mu}\leq\lambda_{6}$.
\end{enumerate}
\end{cor}
\begin{proof}
Follows from the proof of Theorem \ref{thm: ded bigger than exponent}
using Proposition \ref{prop: two ded plus iterations} instead of
Proposition \ref{prop: two ded iterations}.\end{proof}
\begin{problem}
What is the smallest $1<n\leq4$ for which Theorem \ref{thm: ded bigger than exponent}
remains true? Can the bound be improved at least for certain classes
of cardinals? Also, how might the required number of iterations vary
in different models of ZFC?\end{problem}
\begin{cor}
\label{cor: ded bigger then beth} For every cardinal $\mu$ and $k<\omega$
there is some $n<\omega$ and a sequence $\left\langle \lambda_{m}:m\leq n\right\rangle $
such that:
\begin{itemize}
\item \textup{$\lambda_{0}\leq\mu$,}
\item $\lambda_{0}<...<\lambda_{n}$ and $\Ded(\lambda_{m})\,^{+}\geq\lambda_{m+1}$,
\item $\lambda_{n}\geq\beth_{k}\left(\mu\right)$.
\end{itemize}
\end{cor}
\begin{proof}
Follows by iterating Corollary \ref{cor: reaching exp by ded plus}.
\end{proof}

\section{\label{sec: no beth_omega for dependent T}On 2-cardinal models for
dependent $T$}

We recall that a formula $\varphi\left(x,y\right)\in L$ is said to
have the independence property (or IP) with respect to a theory $T$
if in some model of $T$ there are elements $\left\langle a_{i}:i\in\omega\right\rangle $
and $\left\langle b_{s}:s\subseteq\omega\right\rangle $ such that
$\varphi\left(a_{i},b_{s}\right)$ holds if and only if $i\in s$.
A complete first-order theory is called dependent (or NIP) if no formula
has the independence property. The class of dependent theories contains
both the stable and the o-minimal theories, but also for example the
theory of algebraically closed valued fields.
\begin{fact}
\label{fac: NIP iff ded-many types}\cite[Theorem II.4.11]{MR1083551}
A countable theory $T$ is dependent if and only if $\left|S_{1}(M)\right|\leq\left(\ded\left|M\right|\right)^{\aleph_{0}}$
for all $M\models T$.
\end{fact}
In this section we show that when considering the two-cardinal transfer
to arbitrarily large gaps between the cardinals, the situation for
dependent theories is not better than for arbitrary theories. Namely,
for every $n<\omega$ we construct a dependent theory $T$ which has
a $(\beth_{m},\aleph_{0})$-model for all $m<n$, but does not have
any $\left(\beth_{\omega},\aleph_{0}\right)$-models.
\begin{defn}
For any $n\in\mathbb{N}$, let $L_{n}$ be the language consisting
of: 
\begin{enumerate}
\item $P_{m}$, $Q_{m}$ are unary predicates for $m<n$.
\item $f_{m}$ is a unary function for $m+1<n$.
\item $<_{m}$ is a binary relation for $m<n$.
\end{enumerate}
\end{defn}

\begin{defn}
We define a universal theory $T_{n}^{\forall}$ in the language $L_{n}$
saying:
\begin{enumerate}
\item $\left\langle Q_{m}\,:\, m<n\right\rangle $ is a partition of the
universe.
\item $<_{m}$ is a linear order on $Q_{m}$.
\item $P_{m}$ is a subset of $Q_{m}$.
\item $f_{m}$ is a unary function such that:

\begin{enumerate}
\item It is 1-to-1 from $P_{m+1}$ into $Q_{m}\setminus P_{m}$.
\item It is 1-to-1 from $Q_{m}\setminus P_{m}$ into $P_{m+1}$.
\item $f(f(x))=x$. 
\item It is the identity on $\left\{ x\,:\, x\notin P_{m+1}\cup\left(Q_{m}\setminus P_{m}\right)\right\} $.
\end{enumerate}
\end{enumerate}
\end{defn}
\begin{claim}

\begin{enumerate}
\item $T_{n}^{\forall}$ is a consistent universal theory.
\item $T_{n}^{\forall}$ has JEP and AP.
\item If $M\models T_{n}^{\forall}$ and $A\subseteq M$ is finite, then
the substructure generated by $A$ is finite, and in fact of size
at most $2\times\left|A\right|$.
\item $T_{n}^{\forall}$ has a model completion $T_{n}$ which is $\aleph_{0}$-categorical
and eliminates quantifiers.
\end{enumerate}
\end{claim}
\begin{proof}
(1), (2) and (3) are easy to see, and (4) follows by e.g. \cite[Theorem 7.4.1]{Hodges}.\end{proof}
\begin{claim}
\label{cla: axiomatization of T_n} In fact, $T_{n}$ is axiomatized
by:
\begin{enumerate}
\item $T_{n}^{\forall}$
\item $<_{m}$ is a dense linear order without end-points.
\item $P_{m}$ is both dense and co-dense in $Q_{m}$.
\item $f_{m}$ is a 1-to-1 function from $P_{m+1}$ onto $Q_{m}\setminus P_{m}$.
\item If $a_{1}<_{m}c_{1}$ and $a_{2}<_{m+1}c_{2}$, then there are $b_{1}\in Q_{m}\setminus P_{m}$
and $b_{2}\in P_{m+1}$ such that: $a_{1}<_{m}b_{1}<_{m}c_{1}$, $a_{2}<_{m+1}b_{2}<_{m+1}c_{2}$
and $f_{m}(b_{2})=b_{1}$.
\end{enumerate}
\end{claim}

\begin{prop}
\label{prop: T_n is dependent} $T_{n}$ is dependent.\end{prop}
\begin{proof}
Let $M\models T_{n}$. Let $p(x)\in S_{1}(M)$ be a non-algebraic
type. By quantifier elimination it is determined by:
\begin{itemize}
\item $Q_{m}(x)$ for the corresponding $m<n$.
\item Fixing the corresponding cut of $x$ over $M$ in the order $<_{m}$.
\item Saying if $P_{m}(x)$ holds or not.
\item If it doesn't hold, fixing the cut of $f_{m}\left(x\right)$ over
$M$ in the order $<_{m+1}$.
\item If it holds, fixing the cut $f_{m}\left(x\right)$ over $M$ in the
order $<_{m-1}$.
\end{itemize}
Then clearly $\left|S_{1}(M)\right|\leq\ded\left|M\right|$, so $T_{n}$
is dependent.\end{proof}
\begin{rem}
In fact it is easy to check that $T_{n}$ is strongly dependent (see
\cite{Sh863}).\end{rem}
\begin{prop}
\label{prop: no beth omega models}
\begin{enumerate}
\item If $M\models T_{n}$ and $\left|P_{0}^{M}\right|=\lambda$, then $\left|M\right|\leq\beth_{n}(\lambda)$.
\item Moreover: $\left|P_{m+1}^{M}\right|=\left|Q_{m}^{M}\setminus P_{m}^{M}\right|\leq\left|Q_{m}^{M}\right|$
and $\left|Q_{m}^{M}\right|\leq^{+}\Ded\left|P_{m}^{M}\right|$.
\end{enumerate}
\end{prop}

\begin{claim}
\label{cla: choosing a model, successor step} Assume that $\lambda_{0}<\ldots<\lambda_{n}$
and $\lambda_{m+1}\leq^{+}\Ded\lambda_{m}$. Then $T_{n}$ has a model
$M$ such that $\left|P_{0}^{M}\right|=\lambda_{0}$ and :
\begin{enumerate}
\item $\left|P_{m}^{M}\right|=\lambda_{m}$.
\item $\left|Q_{m}^{M}\right|=\lambda_{m+1}$.
\end{enumerate}
\end{claim}
\begin{proof}
By assumption, for every $m<n$ we can find a linear order $J_{m}$
of cardinality $\lambda_{m+1}$ with a dense subset $I_{m}$ of cardinality
$\lambda_{m}$. We may also assume that: 
\begin{enumerate}
\item For every $a<b$ in $J_{m}$, $\left|(a,b)\right|=\lambda_{m+1}$
and $\left|(a,b)\cap I_{m}\right|=\lambda_{m}$ (so in particular
$I_{m}$ is also co-dense in $J_{m}$).
\item $I_{m}$ and $J_{m}$ are dense without end-points.
\end{enumerate}
Indeed, given an arbitrary infinite linear order $I$ and a dense
subset $J$, let $I_{*}=I\times\mathbb{Q}$, $J_{*}=J\times\mathbb{Q}$
and let $I_{**}$ be the lexicographic order on $I_{*}^{<\omega}$,
$J_{**}=J_{*}^{<\omega}$. It is easy to see that $\left|I_{**}\right|=\left|I\right|$,
$\left|J_{**}\right|=\left|J\right|$, $J_{**}$ is dense in $I_{**}$,
both orders are dense without end-points, and that for any $a<b$
in $J_{**}$, $\left|\left(a,b\right)\right|=\left|I\right|$ and
$\left|\left(a,b\right)\cap J_{**}\right|=\left|J\right|$.

We define $M$ by taking $Q_{m}^{M}=J_{m}$, $P_{m}^{M}=I_{m}$ and
$<_{m}^{M}=<_{J_{m}}$. We may choose $f_{m}$ satisfying \ref{cla: axiomatization of T_n}(4)
by transfinite induction as all the relevant intervals have ``full
cardinality'' by the assumption. By Claim \ref{cla: axiomatization of T_n},
$M\models T_{n}$.
\end{proof}

\begin{thm}
For every $n<\omega$ there is a dependent countable theory $T$ which
has a $(\beth_{m},\aleph_{0})$-model for all $m<n$, but does not
have any $\left(\beth_{\omega},\aleph_{0}\right)$-models.\end{thm}
\begin{proof}
Follows by combining Propositions \ref{prop: T_n is dependent}, \ref{prop: no beth omega models},
Claim  \ref{cla: choosing a model, successor step} and Corollary
\ref{cor: ded bigger then beth}.
\end{proof}

\section{\label{sec: Hanf number for omitting types} Hanf number for omitting
types}

Now we elaborate on the previous example, and for every countable
ordinal $\beta<\omega_{1}$ we find a countable ordinal $\alpha_{*}<\omega_{1}$,
a countable theory $T_{\alpha_{*}}$ and a partial type $p(x)$ such
that:
\begin{itemize}
\item there is a model of $T_{\alpha_{*}}$ omitting $p\left(x\right)$
and of size $\geq\beth_{\beta}$,
\item any model of $T_{\alpha_{*}}$ omitting $p\left(x\right)$ is of size
at most $\beth_{\alpha_{*}}$.
\end{itemize}

\begin{defn}
\label{Def: limit theory T_*} Fix an ordinal $\alpha_{*}<\omega_{1}$.
We describe our theory $T_{\alpha_{*}}$.

\begin{enumerate}
\item $\left\langle Q_{\alpha}\left(x\right)\,:\,\alpha\leq\alpha_{*}\right\rangle $
are pairwise disjoint infinite unary predicates.
\item $<_{\alpha}$ is a dense linear order without end-points on $Q_{\alpha}\left(x\right)$.
\item $P_{\alpha}\left(x\right)$ is a dense co-dense subset of $Q_{\alpha}\left(x\right)$.
\item $R\left(x\right)$ is a unary predicate disjoint from all $Q_{\alpha}$'s.
\item $\left\langle c_{n}\,:\, n\in\omega\right\rangle $ are constants
and $R\left(c_{n}\right)$ for all $n\in\omega$.
\item $<_{R}$ is a linear order on $R\left(x\right)$, and $\left(R,<_{R},\left\langle c_{n}\,:\, n\in\omega\right\rangle \right)$
is a model of $\mbox{Th}\left(\mathbb{N},<\,,\left\langle n:n\in\mathbb{N}\right\rangle \right)$.
\item $s_{R}\left(x\right),s_{R}^{-1}\left(x\right)$ are the successor
and the predecessor functions on $R\left(x\right)$.
\item $\left\langle d_{r}\,:\, r\in\mathbb{Q}\right\rangle $ are constants
and $P_{0}\left(d_{r}\right)$ for all $r\in\mathbb{Q}$.
\item For every successor ordinal $\delta+1\leq\alpha_{*}$:

\begin{enumerate}
\item $f_{\delta}$ is a bijection from $P_{\delta+1}$ onto $Q_{\delta}\setminus P_{\delta}$,
identity on $\left\{ x\,:\, x\notin P_{\delta+1}\cup\left(Q_{\delta}\setminus P_{\delta}\right)\right\} $
and such that $f_{\delta}\left(f_{\delta}\left(x\right)\right)=x$.
\item If $a_{1}<_{\delta}c_{1}$ and $a_{2}<_{\delta+1}c_{2}$ for some
$a_{1},c_{1}\in Q_{\delta}\setminus P_{\delta}$ and $a_{2},c_{2}\in P_{\delta+1}$,
then there are $b_{1}\in Q_{\delta}\setminus P_{\delta}$ and $b_{2}\in P_{\delta+1}$
such that: $a_{1}<_{\delta}b_{1}<_{\delta}c_{1}$, $a_{2}<_{\delta+1}b_{2}<_{\delta+1}c_{2}$
and $f_{\delta}(b_{2})=b_{1}$.
\end{enumerate}
\item For every limit ordinal $\delta\leq\alpha_{*}$:

\begin{enumerate}
\item We fix some listing $\left\langle \alpha_{\delta,n}\,:\, n<\omega\right\rangle $
with $\sum_{n<\omega}\alpha_{\delta,n}=\delta$, where for every $n$
we have that $\alpha_{\delta,n}$ is a\emph{ successor} ordinal larger
than the successor of $\alpha_{\delta,n-1}$ and larger than any $\alpha_{\delta',m}$
from a similar listing for a smaller limit ordinal $\delta'$.
\item We have a function $G_{\delta}\left(x\right)$ such that:

\begin{enumerate}
\item $G_{\delta}$ is the identity on $\left\{ x\,:\, x\notin P_{\delta}\right\} $.
\item $G_{\delta}:\, P_{\delta}\left(x\right)\to R\left(x\right)$ is onto.
\item for every $y\in R\left(x\right)$, $G_{\delta}^{-1}\left(y\right)$
is a dense linear order without end-points.
\item If $y_{1}<_{R}y_{2}$, then $G_{\delta}^{-1}\left(y_{1}\right)$ is
co-dense in $G_{\delta}^{-1}\left(y_{2}\right)$, and every cut of
$G_{\delta}^{-1}\left(y_{1}\right)$ realized by some $a\in P_{\delta}$
is realized by some $a'\in G_{\delta}^{-1}\left(y_{2}\right)$.
\end{enumerate}
\item We have a relation $E_{\delta}\left(x_{1},x_{2},y\right)$ which holds
if and only if $x_{1}$ and $x_{2}$ are from $P_{\delta}\setminus G_{\delta}^{-1}\left(y\right)$
and realize the same cut over $G_{\delta}^{-1}\left(y\right)$.
\item For each $n\in\omega$ we have a function $F_{\delta,n}$ such that:

\begin{enumerate}
\item It is a bijection from $G_{\delta}^{-1}\left(c_{n}\right)\setminus G_{\delta}^{-1}\left(c_{n-1}\right)$
onto $P_{\alpha_{\delta,n}}\left(x\right)$, the identity on $\{x\,:\, x\notin P_{\alpha_{\delta,n}}\cup G_{\delta}^{-1}\left(c_{n}\right)\}$
and such that $F_{\delta,n}\left(F_{\delta,n}\left(x\right)\right)=x$.
\item For any $n\in\omega$, if $a_{1}<_{\alpha_{\delta,n}}b_{1}$ with
$a_{1},b_{1}\in P_{\alpha_{\delta,n}}$ and $a_{2}<_{\delta}d<_{\delta}b_{2}$
with $a_{2},b_{2}\in G_{\delta}^{-1}\left(c_{n}\right)$, then there
are $e_{1}\in P_{\alpha_{\delta,n}}$ and $e_{2}\in G_{\delta}^{-1}\left(c_{n}\right)\setminus G_{\delta}^{-1}\left(c_{n-1}\right)$
such that: $a_{1}<_{\delta}e_{1}<_{\delta}b_{1}$, $a_{2}<_{\delta}e_{2}<_{\delta}b_{2}$,
$F_{\delta,n}(e_{2})=e_{1}$ and $E_{\delta}\left(d,e_{2},\alpha\right)$
for all $\alpha<c_{n}$.
\end{enumerate}
\end{enumerate}
\end{enumerate}
\end{defn}
\begin{claim}
\label{cla: T_* is dependent}$T_{\alpha_{*}}$ is a complete dependent
theory.\end{claim}
\begin{proof}
It it easy to check by back-and-forth that $T$ is a complete theory
eliminating quantifiers.

Let $M\models T_{\alpha_{*}}$ and let $p\left(x\right)\in S_{1}\left(M\right)$
be a non-algebraic type. We have the following options:
\begin{enumerate}
\item $p\left(x\right)\vdash Q_{\alpha}\left(x\right)$ for some successor
$\alpha<\alpha_{*}$. Then $p\left(x\right)$ is determined by:

\begin{enumerate}
\item Fixing the cut of $x$ over $M$ in the order $<_{\alpha}$.
\item If $p\left(x\right)\vdash\neg P_{\alpha}(x)$:

\begin{enumerate}
\item Fixing the cut of $f_{\alpha}\left(x\right)$ over $M$ in the order
$<_{\alpha+1}$.
\item If $\alpha+1$ occurs as $\alpha_{\delta,n}$ for some limit $\delta<\alpha_{*}$,
then fixing the cut of $F_{\delta,n}\left(f_{\alpha}\left(x\right)\right)$
over $M$ in the order $<_{\delta}$, and fixing the cut of $G_{\delta}\left(F_{\delta,n}\left(f_{\alpha}\left(x\right)\right)\right)$
in $<_{R}$ over $M$.
\end{enumerate}
\item If $p\left(x\right)\vdash P_{\alpha}(x)$: 

\begin{enumerate}
\item fixing the cut $f_{\alpha-1}\left(x\right)$ over $M$ in the order
$<_{\alpha-1}$.
\item If $\alpha$ occurs as $\alpha_{\delta,n}$ for some limit $\delta<\alpha_{*}$,
then fixing the cut of $F_{\delta,n}\left(x\right)$ over $M$ in
the order $<_{\delta}$, and fixing the cut of $G_{\delta}\left(F_{\delta,n}\left(x\right)\right)$
in $<_{R}$ over $M$.
\end{enumerate}
\end{enumerate}
\item $p\left(x\right)\vdash Q_{\delta}\left(x\right)$ for some limit $\delta$.
Then $p\left(x\right)$ is determined by:

\begin{enumerate}
\item Fixing the cut of $x$ over $M$ in the order $<_{\delta}$.
\item If $P_{\delta}\left(x\right)$ does not hold, then similar to 2(b).
\item If $P_{\delta}\left(x\right)$ holds:

\begin{enumerate}
\item Fixing the cut of $G_{\delta}\left(x\right)$ over $M$ in $<_{R}$.
\item If $G_{\delta}\left(x\right)=c_{n}$ for some $n\in\omega$ also fixing
the cut of $F_{\delta,n}\left(x\right)$ over $M$ in $<_{\alpha_{\delta,n}}$.
\end{enumerate}
\end{enumerate}
\item If $p\left(x\right)\vdash R\left(x\right)$, then fixing the cut of
$x$ in $<_{R}$ over $M$.
\item $p\left(x\right)\vdash\left\{ \neg Q_{\alpha}\left(x\right)\,:\,\alpha<\alpha_{*}\right\} \cup\left\{ \neg R\left(x\right)\right\} $.
Then $p\left(x\right)$ is a complete type.
\end{enumerate}
Altogether it follows that $\left|S_{1}\left(M\right)\right|\leq\left(\ded\left|M\right|\right)^{\aleph_{0}}$,
thus $T$ is dependent by Fact \ref{fac: NIP iff ded-many types}.
\end{proof}
~

Consider the type $p_{*}(x)=\left\{ \neg P_{\alpha}(x)\,:\,0<\alpha\leq\alpha_{*}\right\} \cup\left\{ x\neq c_{n}\,:\, n\in\omega\right\} \cup\left\{ x\neq d_{r}\,:\, r\in\mathbb{Q}\right\} $.
\begin{claim}
\label{claim: models of T_* are not too big} Let $M$ be a model
of $T_{\alpha_{*}}$ omitting $p_{*}\left(x\right)$. Then $\left|M\right|\leq\beth_{\alpha_{*}}$.\end{claim}
\begin{proof}
First of all, if $M$ omits $p_{*}$ then $\left|P_{0}^{M}\right|=\aleph_{0}$
and $\left|R^{M}\right|=\aleph_{0}$. We show by induction for $\delta\leq\alpha_{*}$
that $\left|P_{\delta}^{M}\right|\leq\beth_{\delta}$. If $\delta=\alpha+1$
is a successor, then clearly $\left|P_{\delta+1}^{M}\right|\leq^{+}\ded\left|P_{\delta}^{M}\right|$,
thus $\leq\beth_{\delta+1}$ by induction. If $\delta$ is a limit,
then by construction $\left|P_{\delta}^{M}\right|\leq\sum_{n<\omega}\left(\left|P_{\alpha_{\delta,n}}^{M}\right|\right)\leq\sum_{n<\omega}\beth_{\alpha_{\delta,n}}=\beth_{\delta}$.
The claim follows.\end{proof}
\begin{claim}
\label{cla: choosing a model of T_*} For every $\beta<\omega_{1}$
there is $\alpha_{*}<\omega_{1}$ such that $T_{\alpha_{*}}$ has
a model omitting $p_{*}\left(x\right)$ of size $\geq\beth_{\beta}$.\end{claim}
\begin{proof}
By Corollary \ref{cor: ded bigger then beth} and induction there
is $\alpha_{*}<\beta+\omega$ such that we can choose a strictly increasing
sequence of cardinals $\left(\lambda_{\alpha}\right)_{\alpha<\alpha_{*}}$
satisfying:
\begin{itemize}
\item $\lambda_{0}=\aleph_{0}$.
\item $\lambda_{\alpha+1}\leq^{+}\ded\lambda_{\alpha}$.
\item For a limit $\alpha$, $\lambda_{\alpha}=\sum_{\alpha'<\alpha}\lambda_{\alpha'}$.
\item $\lambda_{\alpha_{*}}\geq\beth_{\beta}$.
\end{itemize}
We define a model of $T_{\alpha_{*}}$ omitting $p_{*}$ and such
that $\left|P_{\alpha}^{M}\right|=\lambda_{\alpha}$ by induction
on $\alpha$.
\begin{enumerate}
\item Let $R^{M}=\left(\omega,<\right)$ with $c_{n}$ naming $n$. Let
$Q_{0}^{M}=\left(\mathbb{R},<\right)$ and let $P_{0}^{M}=\mathbb{Q}$,
with $d_{r}$ naming $r$.
\item For a successor $\delta=\alpha+1$: Similarly to Claim \ref{cla: choosing a model, successor step},
we can find a linear order $J$ of cardinality $\lambda_{\delta}$
with a dense subset $I$ of cardinality $\lambda_{\alpha}$. We may
also assume that for every $a<b$ in $J$, $\left|(a,b)\right|=\lambda_{\delta}$
and $\left|(a,b)\cap I\right|=\lambda_{\alpha}$. We let $Q_{\delta}^{M}=J$,
$P_{\delta}^{M}=I$ and $<_{\delta}^{M}=<_{J}$. We may choose $f_{\delta}$
satisfying Definition \ref{Def: limit theory T_*} by transfinite
induction as all the relevant intervals have ``full cardinality''
by construction and the inductive assumption.
\item For a limit $\delta\leq\alpha_{*}$:

\begin{enumerate}
\item First we construct orders $I_{n},J_{n}$ by induction on $n<\omega$:

\begin{enumerate}
\item Let $I_{0}\subseteq J_{0}$ be dense linear orders without end-points
and such that $I_{0}$ is dense-codense in $J_{0}$, $\left|I_{0}\right|=\lambda_{\alpha_{\delta,0}}$,
$\left|J_{0}\right|=\lambda_{\alpha_{\delta,0}+1}$, and such that
for every $a<b$ in $J_{0}$, $\left|(a,b)\right|=\lambda_{\alpha_{\delta,0}+1}$
and $\left|(a,b)\cap I_{0}\right|=\lambda_{\alpha_{\delta,0}}$ (can
be chosen by assumption on $\lambda_{\alpha}$ as in the proof of
Claim \ref{cla: choosing a model, successor step}). 
\item Let $I_{n+1}',J_{n+1}'$ be dense linear orders without end-points
and such that $I_{n+1}'$ is dense-codense in $J_{n+1}'$, $\left|I_{n+1}'\right|=\lambda_{\alpha_{\delta,n+1}}$,
$\left|J_{n+1}'\right|=\lambda_{\alpha_{\delta,n+1}+1}$, and such
that for every $a<b$ in $J_{n+1}'$, $\left|(a,b)\right|=\lambda_{\alpha_{\delta,n+1}+1}$
and $\left|(a,b)\cap I_{n+1}'\right|=\lambda_{\alpha_{\delta,n+1}}$
(again can be chosen by assumption on $\lambda_{\alpha}$ as in the
proof of Claim \ref{cla: choosing a model, successor step}). Let
$I_{n+1}$ extend $I_{n}$ with a copy of $I_{n+1}'$ added in every
cut, and similarly let $J_{n+1}$ extend $J_{n}$ with a copy of $J_{n+1}'$
added in every cut. It follows that $\lambda_{\delta,n+1}\leq|I_{n+1}|\leq\lambda_{\alpha_{\delta,n}+1}\times\lambda_{\alpha_{\delta,n+1}}\leq\lambda_{\alpha_{\delta,n+1}}$
and $\left|J_{n+1}\right|\leq\lambda_{\alpha_{\delta,n}+2}\times\lambda_{\alpha_{\delta,n+1}+1}\leq\lambda_{\alpha_{\delta,n+1}+1}$,
and that $I_{n+1}$ is a dense-codense subset of $J_{n+1}$.
\item Finally, let $I=\bigcup_{n<\omega}I_{n}$ and $J=\bigcup_{n<\omega}J_{n}$.
In particular $I$ is dense-codense in $J$ and both $I,J$ are of
size $\lambda_{\delta}$.
\end{enumerate}
\item We let $P_{\delta}^{M}=I,Q_{\delta}^{M}=J$ and define $G_{\delta}^{M}$
by sending $I_{n}$ to $c_{n}$. By construction of $I_{n}$ and $P_{\alpha_{\delta,n}}^{M}$
and transfinite induction we can find bijections $F_{\delta,n}^{M}$
between $G_{\delta}^{M}\left(c_{n}\right)\setminus G_{\delta}^{M}\left(c_{n-1}\right)=I_{n}\setminus I_{n-1}$
and $P_{\alpha_{\delta,n}}^{M}$ satisfying the axioms of $T_{\alpha_{*}}$.
We let $E\left(x,y,c_{n}\right)$ hold for $x,y$ in $I_{n}\setminus I_{n-1}$
realizing the same cut over $I_{n-1}$.
\end{enumerate}
\end{enumerate}
\end{proof}
\begin{thm}
For every countable ordinal $\beta<\omega_{1}$ there is a complete
countable dependent theory $T$ and a partial type $p(x)$ such that:
\begin{itemize}
\item $T$ has a model omitting $p$ of size $\geq\beth_{\beta}$.
\item Any model of $T$ omitting $p$ is of size $<\beth_{\omega_{1}}$.
\end{itemize}
\end{thm}
\begin{proof}
Combining Claims \ref{cla: T_* is dependent}, \ref{claim: models of T_* are not too big}
and \ref{cla: choosing a model of T_*}.
\end{proof}

\bibliographystyle{alpha}
\bibliography{common}

\end{document}